\documentclass[12pt]{article}
\usepackage{amsmath,amsthm,verbatim,amssymb,amsfonts,amscd,diagrams, graphics, mathrsfs}
\theoremstyle{plain}
\newtheorem{theorem}{Theorem}[section]
\newtheorem{corollary}[theorem]{Corollary}
\newtheorem{lemma}[theorem]{Lemma}
\newtheorem{proposition}[theorem]{Proposition}

\theoremstyle{definition}

\newtheorem*{question}{Question}

\theoremstyle{remark}
\newtheorem{remark}[theorem]{Remark}

\newarrow{ul}---->
\newarrow{Backwards}<----

\newcommand{\into}{\hookrightarrow}
\newcommand{\Z}{\mathbb{Z}}
\newcommand{\Q}{\mathbb{Q}}
\newcommand{\R}{\mathbb{R}}

\newcommand{\F}{\mathbb{F}}
\newcommand{\C}{\mathbb{C}}

\newcommand{\bd}{\partial}

\newcommand{\mc}[1]{\mathcal{#1}}

\newcommand{\dlim}{\varinjlim}
\newcommand{\ilim}{\varprojlim}

\newcommand{\mf}{\mathfrak}

\newarrow{Onto}----{>>}
\newarrow{Equals}=====
\newarrow{Into}C--->

\begin{document}

\title{Intersection homology with field coefficients: \\ $K$-Witt spaces and $K$-Witt bordism}
\author{Greg Friedman\\Texas Christian University}
\date{April 16, 2008}
\maketitle

\begin{abstract}
We construct geometric examples of pseudomanifolds that satisfy the Witt condition for intersection homology Poincar\'e duality with respect to certain fields but not others. We also compute the bordism theory of $K$-Witt spaces for an arbitrary field $K$, extending results of Siegel for $K=\Q$.  
\end{abstract}

\textbf{2000 Mathematics Subject Classification: }  55N33, 57Q20, 57N80 

\textbf{Keywords: } intersection homology, universal coefficient theorem, Witt space, Witt bordism

\tableofcontents

\section{Introduction}

This paper consists of two related parts: In the first part, we provide some examples of the phenomena that arise when considering intersection homology over coefficient groups with torsion, including various forms of violation of the universal coefficient theorem and spaces that satisfy the \emph{Witt condition} for certain fields but not others (and hence possess Poincar\'e duality with respect to certain fields but not others). In the second part, we compute the bordism theory of spaces that are $K$-Witt for an arbitrary field $K$. Let us explain what all this means by providing some rough definitions and context; more precise background details can be found below in Section \ref{S: background}.

Intersection homology groups were developed by Goresky and MacPherson in \cite{GM1} for the purpose of extending Poincar\'e duality and ensuing invariants to non-manifold spaces. In \cite{GM1}, this is accomplished, \emph{over rational coefficients}, for the class of (compact, oriented) piecewise linear pseudomanifolds, a class of spaces including all projective complex varieties. Using sheaf theory, this duality was later expanded \cite{GM2} to the broader class of topological pseudomanifolds and to coefficients in any field. Other generalizations followed; see, e.g. \cite{GBF17}.

The dual pairing of intersection homology, in generality,  pairs intersection homology groups with dual indices (as in the familiar case of manifolds) \emph{and} with dual sets of perversity parameters. Thus the general duality result  for an $n$-dimensional pseudomanifold asserts that there is a perfect (nonsingular) pairing

$$I^{\bar p}H_i(X;K)\times I^{\bar q}H_{n-i}(X;K)\to K.$$
Here $IH$ denotes the intersection homology groups, $K$ is a coefficient field, and $\bar p$ and $\bar q$ are dual sets of \emph{perversity parameters} that occur as one of the inputs to the theory; see Section \ref{S: background}, below, for details. 

Ideally, however, one would like a little more. For $2k$-dimensional spaces, one would like a $(-1)^k$-symmetric self-pairing on $I^{\bar p}H_k(X;K)$. This would yield signatures, elements of Witt groups, and other further algebraic information. Unfortunately, this is not possible in general, but there are dual ``middle perversities'', $\bar m$ and $\bar n$, and certain spaces such that $I^{\bar m}H_i(X;K)\cong I^{\bar n}H_i(X;K)$, in which case we obtain the desired middle dimensional form. 

It was recognized early on (right in \cite{GM1}) that pseudomanifolds with only even codimension singularities possess this form of self-duality. This class was soon generalized by Siegel \cite{Si83} to a class of spaces he dubbed Witt spaces, and which we will more specifically call $\Q$-Witt spaces. These spaces are identified by certain local intersection homology conditions, and they possess the middle-dimensional self-duality over $\Q$. Siegel further computed the bordism groups of these spaces, showing that in nontrivial cases they equal the Witt group $W(\Q)$ - hence the name ``Witt spaces'' - and that the resulting bordism homology theory provides  a geometric formulation of $KO$-homology at odd primes.

Banagl \cite{Ba02} has since extended duality even further by identifying conditions on which non-Witt spaces possess self-duality (conditions equivalent to the existence of certain towers of Lagrangian structures on strata), but Witt spaces remain an important class of examples defined by a relatively tractable condition. 

This brings us to intersection homology with coefficients. Unlike ordinary homology theories, intersection homology does not, in general,  possess a universal coefficient theorem  (though Goresky and Siegel \cite{GS83} have shown, using the Deligne sheaf formulation of intersection homology, that a universal coefficient sequence will occur if a space possesses certain local torsion properties). This has not prevented important work employing intersection homology with coefficient fields of finite characteristic, for instance there is a version of the Weil conjecture for singular varieties using $\ell$-adic intersection homology (see \cite[Chapter 10]{KirWoo}). However, intersection homology groups with different coefficients must be treated in their own right, without any clear connections between them. In particular, spaces that satisfy intersection homology Poincar\'e duality with one set of coefficients may not possess duality with respect to other sets of coefficients. 

Our first goal is to provide some examples of these phenomena. We produce concrete examples of spaces where the universal coefficient theorem breaks down (in different ways), and we present spaces that are $K$-Witt (and hence possess self-duality) with respect to some coefficient fields $K$ but not others. Our arguments and constructions are purely geometric, avoiding sheaf theory in favor of hands-on examination of intersection chains. 

 The following facts will be demonstrated throughout Section 
\ref{S: spaces}
 (except for the first, which is shown in Section \ref{S: reduction}):

\begin{theorem}\label{T: everything}
\begin{enumerate}
\item  If $K$ has characteristic $p>0$, then $X$ is $K$-Witt if and only if $X$ is $\Z_p$-Witt; if $K$ has characteristic $0$, then $X$ is $K$-Witt if and only if $X$ is $\Q$-Witt.

\item \label{I: item} If $n>4$ and $P$ is a finite set of primes, then there is a compact orientable $n$-dimensional pseudomanifold that is $\Z_p$-Witt for any $p\in P$ but that is not  $\Q$-Witt and  not  $\Z_{p}$-Witt for $p\notin P$.

\item \label{I: Qwitt} If $n>4$ and $P$ is a finite set of  primes, then there are $\Q$-Witt spaces that are not $\Z_p$-Witt for any $p\in P$ and are $\Z_p$-Witt for $p\notin P$.

\item \label{I: 3-4.1} If $X$ is a $3$- or $4$-dimensional $\Z_p$-Witt space, then $X$ is a $\Q$-Witt space. 

\item \label{I: 3-4.2} If $X$ is a $3$- or $4$-dimensional  $\Q$-Witt space, then $X$ is a $\Z_p$-Witt space for any $p\neq 2$. If $X$ is also $\Q$-orientable, then it is also a $\Z_2$-Witt space. However, there are non-orientable $3$- and $4$-dimensional PL  $\Q$-Witt spaces that are not $\Z_2$-Witt spaces. 

\item \label{I: 0-1-2} All $0$-, $1$-, and $2$-dimensional pseudomanifolds are $K$-Witt for all $K$.
\end{enumerate}
\end{theorem}

We also find examples of $4k$-dimensional pseudomanifolds demonstrating conditions \eqref{I: item} and \eqref{I: Qwitt} that not just satisfy or fail to satisfy the appropriate Witt conditions but that also definitively possess or fail to possess the associated dualities in nontrivial ways.

Finally, in the second half of the paper, Section \ref{S: bordism}, we follow Siegel \cite{Si83} by computing the bordism groups $\Omega_n^{K-\text{Witt}}$ of oriented $K$-Witt spaces for any coefficient field $K$ as well as identifying the resulting generalized homology theories. We show the following theorems (bear in mind item (1) of the preceding theorem, which implies that $\Omega_n^{K-\text{Witt}}\cong \Omega_n^{K'-\text{Witt}}$ if $K$ and $K'$ have the same characteristic):

\begin{theorem}[Theorem \ref{T: F-Witt}]\label{T: witt1}
Let $p\neq 0$ be a prime, and let $W(\Z_p)$ denote the Witt group of symmetric bilinear forms over $\Z_p$.  
\begin{enumerate}
\item $\Omega_0^{\Z_p-\text{Witt}}\cong \Z$. 

\item For $n\not \equiv 0\mod 4$, $\Omega_n^{\Z_p-\text{Witt}}=0$. 

\item For $n\equiv 0\mod 4$, $n>0$, the homomorphism $w:\Omega_n^{\Z_p-\text{Witt}}\to W(\Z_p)$ that assigns to $X^n$ the intersection form on $ I^{\bar m}H_{n/2}(X;\Z_p)$ is an isomorphism.
\end{enumerate}
\end{theorem}

\begin{theorem}[Theorem \ref{T: EM}]\label{T: witt2}
As a homology theory on CW complexes, $\Z_p$-Witt bordism splits as a direct sum of ordinary homology with coefficients. In particular, $$\Omega^{\Z_p-\text{Witt}}_n(X)\cong  \bigoplus_{r+s=n} H_r(X;\Omega^{\Z_p-\text{Witt}}_s).$$ 
\end{theorem}

\paragraph{Acknowledgment.} I thank Shmuel Weinberger for suggesting that it would be interesting to study  intersection homology with coefficients, in general, and for asking about extending Siegel's results to $\Z_p$-Witt bordism, in particular. 
In the course of chasing down Theorem \ref{T: EM}, I had the pleasure of several fruitful correspondences and therefore owe great thanks to Jim McClure, Andrew Ranicki, Yuli Rudyak, Larry Taylor, and Shmuel Weinberger.

\section{Background material}\label{S: background}

In this section we provide the relevant background for the rest of the paper.

\paragraph{Pseudomanifolds.} We will work entirely in the class of piecewise linear (PL) spaces, although intersection homology can be defined more broadly on topological pseudomanifolds (see \cite{GM2}). 

We recall (see \cite{GM1}) that a PL stratified pseudomanifold $X$ is a PL space
equipped with a filtration (compatible with the PL structure)
\begin{equation*}
X=X^n\supset X^{n-2} \supset X^{n-3}\supset \cdots \supset X^0\supset X^{-1}=\emptyset
\end{equation*}
such that for each point $x\in X_i=X^i-X^{i-1}$, there exists 
a lower-dimensional compact PL stratified pseudomanifold $L$, a compatible filtration  of $L$
\begin{equation*}
L=L^{n-i-1}\supset  \cdots \supset L^0\supset L^{-1}=\emptyset,
\end{equation*}
and 
a \emph{distinguished neighborhood}
$U$ of $x$ such that there is 
a PL homeomorphism
\begin{equation*}
\phi: \R^i\times c(L)\to U
\end{equation*}
that takes $\R^i\times c(L^{j-1})$ onto $X^{i+j}\cap U$. Here $cL$ denotes the open cone on $L$. In other words, each point has a neighborhood that is a trivial bundle of cones on a lower-dimensional stratified space. 

The subspace $X_i=X^i-X^{i-1}$ is called the $i$th \emph{stratum}, and, in particular, it is a (possibly empty) PL $i$-manifold.  $L$ is called the \emph{link} of the component of the stratum. Note that we do not allow a codimension $1$ stratum. There are various technical reasons for this, including the avoidance of ``pseudomanifolds with boundary'' (see, e.g., \cite{GBF10,GBF11}, where this issue is treated in detail); however, we will revisit this idea in our discussion of bordism in Section \ref{S: bordism}. $X^{n-2}$ is often referred to as the \emph{singular locus} and denoted $\Sigma$. A PL stratified pseudomanifold $X$ is \emph{oriented} if $X-\Sigma=X-X^{n-2}$ is oriented as a manifold.

\paragraph{Intersection homology.} Intersection homology, due to Goresky and MacPherson \cite{GM1}, is a topological invariant of pseudomanifolds (in particular, it is invariant under choice of PL structure or stratification - see \cite{GM2}, \cite{Bo}, \cite{Ki}). It possesses a definition via sheaf theory, which is important (indeed crucial) for many applications, but the original definition was given as the homology of  a subcomplex of the complex $C_*(X)$ of PL chains on $X$. This $C_*(X)$ is a direct limit $\varinjlim_{T\in \mc T} C_*^T(X)$, where $C_*^T(X)$ is the simplicial chain complex with respect to the triangulation $T$, and  the direct limit is taken with respect to subdivision within a family of triangulations compatible with each other under subdivision and compatible with the filtration of $X$. In fact, while it is convenient to work with these PL chains, one can also work with simplicial chains, supposing a fine enough triangulation of $X$ (see the appendix to \cite{MV86}).

Intersection chain complexes are subcomplexes of $C_*(X)$ defined with regard to  \emph{perversity parameters} $\bar p:\Z^{\geq 2}\to  \Z$ that are required to satisfy $\bar p(2)=0$ and $\bar p(k)\leq \bar p(k+1)\leq \bar p(k)+1$. We think of the perversity as taking the codimensions of the strata of $X$ as input. The output tells us the extent to which chains in the intersection chain complex will be allowed to intersect that stratum. Thus a simplex $\sigma$ in $C_i(X)$ (represented by a simplex in some triangulation) is deemed $\bar p$-\emph{allowable} if  $\dim(\sigma\cap X^{n-k})\leq i-k+\bar p(k)$, and a chain $\xi\in C_i(X)$ is $\bar p$-allowable if every simplex with non-zero coefficient in $\xi$ or $\bd \xi$  is allowable as a simplex. The allowable chains constitute the chain complex $I^{\bar p}C_*(X)$, and the $\bar p$-perversity intersection homology groups $I^{\bar p}H_*(X)$ are the homology groups of this chain complex. Note that if $M$ is a manifold, then $I^{\bar p}H_*(M)\cong H_*(M)$. This is not obvious if $M$ is stratified in an interesting way, but it follows from the topological invariance of intersection homology groups, which implies that $I^{\bar p}H_*(M)$ may be computed from the trivial filtration $M\supset \emptyset$.

For more general background on intersection homology, we urge the reader to consult the expositions by Borel, et. al. \cite{Bo} or Banagl \cite{BaIH}. For both background and application of intersection homology in various fields of mathematics, the reader should see Kirwan and Woolf \cite{KirWoo}.

\paragraph{Intersection homology with coefficients.}
The definition of intersection homology with coefficients is given analogously so that $I^{\bar p}C_*(X;G)$ is the subcomplex of $C_*(X;G)\cong C_*(X)\otimes G$, again consisting of chains $\xi$ such that every simplex  with non-zero coefficient in $\xi$ or $\bd \xi$  is allowable as a simplex. However, a critical point to observe is that, in general, $I^{\bar p}C_*(X;G)$ is not isomorphic to $I^{\bar p}C_*(X)\otimes G$. It is true that a simplex with nonzero coefficient in a chain $\xi\in I^{\bar p}C_*(X;G)$ is allowable or not depending only on the simplex itself and not the coefficient. However,  which simplices appear with non-zero coefficient in $\bd \xi$ might depend strongly on the coefficients being used.

 For example, consider a chain of the form $\xi= \sum_i \sigma_i$ over some collection of oriented simplices, each with coefficient $1$. The allowability of each of the simplices $\sigma_i$ is independent of whether we think of $\xi$ as a chain in $C_i(X)$ or $C_i(X;\Z_2)$. Now suppose each $\sigma_i$ in $\xi$ is allowable, and consider $\bd \xi$. Suppose that $\bd \xi=2\eta$ for some chain $\eta$. It is possible in $C_i(X)$ that $\eta$ may contain simplices that are not allowable. However, in $C_i(X;\Z_2)$, $\bd \xi=0$, and the allowability conditions are satisfied vacuously. 
 
 Thus when working with coefficients, the obvious homomorphism $I^{\bar p}C_*(X)\otimes G\to I^{\bar p}C_*(X;G)$ is an injection, but it is not, in general, a surjection. 
These considerations, of course, have the potential to affect the intersection homology groups quite radically. For example, the universal coefficient theorem is not generally valid for intersection homology. In the next section, we turn to concrete examples that demonstrate geometrically what can go wrong.

\paragraph{The cone formula.} Perhaps the most important concrete computation in intersection homology is the formula for the intersection homology of an open cone. If $L$ is an $n$-dimensional compact pseudomanifold, then the open cone $cL$ is stratified so that $(cL)^0$ is the cone vertex and, for $i>0$, $(cL)^i=L^{i-1}\times (0,1)\subset cL$. Then the intersection homology of the cone $cL$ is given as follows:
 \begin{equation*}
I^{\bar p}H_i(cL;G) \cong
\begin{cases}
0, & i\geq n-\bar p(n+1),\\
I^{\bar p}H_{i}(L;G), & i<n-\bar p(n+1).
\end{cases}
\end{equation*}
This formula comes from direct consideration of the definition of the intersection chain complex and the fact that the dimension of the intersection of a simplex with the cone vertex can be at most $0$. See \cite[Section 1]{Bo} for more details.

It is also useful to have the formula for the intersection homology of a suspension, which comes from the cone formula and a Mayer-Vietoris argument (see\footnote{The formula presented here is slightly simpler than the one in \cite{Ki} since we  allow  only traditional perversities, not ``loose'' perversities.} \cite{Ki}). If $X$ is a compact $n$-dimensional pseudomanifold with suspension $SX$, then 
\begin{equation*}
I^{\bar p}H_i(SX;G) \cong
\begin{cases}
I^{\bar p}H_{i-1}(X;G), & i> n-\bar p(n+1),\\
0,& i=n-\bar p(n+1),\\
I^{\bar p}H_{i}(X;G), & i<n-\bar p(n+1).
\end{cases}
\end{equation*}

\paragraph{Witt spaces.} 
The chief interest (at least originally) in intersection homology is that, with field coefficients, it satisfies Poincar\'e Duality. More specifically, assume that $X^n$ is a compact, oriented, and \emph{irreducible} (meaning $X-\Sigma$ is connected) PL pseudomanifold, and let 
 $F$ be  a field  and $\bar p$ and $\bar q$ dual perversities (meaning that $\bar p(k)+\bar q(k)=k-2$ for all $k$).  Then there is a nondegenerate pairing $I^{\bar p}H_i(X;F)\otimes I^{\bar q}H_{n-i}(X;F)\to F$, defined via the intersection pairing on intersection chains in general position. We refer the reader to \cite{GM1, GM2, Bo, GBF18} for more details. While this is good, one would like something even better, a condition that guarantees a self-pairing between middle-dimensional intersection homology groups for even-dimensional manifolds. This is what the Witt spaces provide. 

Let $\bar m$ and $\bar n$ be the lower and upper middle perversities given by $(0,0,1,1,2,2, \dots)$ and $(0,1,1,2,2,3,\ldots)$, i.e. $\bar m(k)=\lfloor \frac{k-2}{2}\rfloor$ and $\bar n(k)=\lfloor \frac{k-1}{2}\rfloor$. These are dual perversities, and it is not hard to check from the definitions that if $X^{2n}$ is compact and oriented and has nonempty strata only of even dimension, then $I^{\bar m}H_*(X;F)\cong I^{\bar n}H_*(X;F)$. So in this case there is a $(-1)^n$-symmetric form  $I^{\bar m}H_n(X;F)\otimes I^{\bar m}H_{n}(X;F)\to F$. When $X$ has dimension $4n$ this yields signatures, etc; see \cite{GM1}. A weaker condition on $X^n$ that yields the same outcome is the \emph{$F$-Witt condition}, which assumes that $I^{\bar m}H_k(L;F)=0$ for each link $L^{2k}$ of each stratum of dimension $n-2k-1$, $k>0$. In this case, it follows once again that $I^{\bar m}H_*(X;F)\cong I^{\bar n}H_*(X;F)$ (see \cite{Bo}), and we obtain middle dimensional pairings \cite{GM2}. 

In keeping with the conventions of \cite{GM1} and \cite{Si83}, we will call an \emph{oriented compact} irreducible PL stratified pseudomanifold satisfying the $F$-Witt condition an \emph{$F$-Witt space}. The orientation condition is implicit in \cite{Si83} based on the definition of pseudomanifold given in \cite{GM1}. 
If we need to refer to a nonorientable pseudomanifold satisfying the Witt condition, we will call it explicitly a ``non-orientable Witt space.''

\section{Oddities of finite coefficients}

In this section, we begin with some simple examples of the violation of the universal coefficient theorem for intersection homology. We move on to more complex examples that are then used to construct spaces that satisfy Witt conditions with respect to certain fields but not others. 

\subsection{Violations of the universal coefficient theorem}\label{S: worse}

In \cite{GS83}, Goresky and Siegel used sheaf machinery to prove that $\bar p$-perversity intersection homology satisfies the universal coefficient theorem for an abelian group $G$ on a pseudomanifold $X$ if $X$  is \emph{locally $\bar p$-torsion free}. This condition means that if $L$ is the link of a stratum of $X$ of codimension $c$, then the abelian group $I^{\bar p}H_{c-2-\bar p(c)}(L)$ is torsion free. While the proof  of the theorem in \cite{GS83} involves the axiomatic sheaf formulation of intersection homology, one can work directly with chains to find examples of  the trouble that can arise if this torsion condition is violated. In this section, we provide several such examples of varying degrees of complexity.

\paragraph{A simple example of violation of universal coefficients.}

As a first example of the violation of the universal coefficient theorem in intersection homology, consider $X=c(\R P^2)$, the open cone on $\R P^2$, and suppose that $\bar p(3)=0$. The link of the singular vertex $v$ of $X$ is $L=\R P^2$, and $I^{\bar p}H_{1}(L)=H_1(L)=\Z_2$. So $X$ is not locally $\bar p$-torsion free. 

We compute from the cone formula (see Section \ref{S: background}):

\begin{equation*}
I^{\bar p}H_i(c\R P^2)\cong
\begin{cases}
0, & i\geq 2-\bar p(3)=2,\\
I^{\bar p}H_i(\R P^2), & i<2.
\end{cases}
\end{equation*}
In particular, $I^{\bar p}H_1(X)\cong \Z_2$ and $I^{\bar p}H_2(X)\cong 0$.

Similarly, since the  cone formula holds for any coefficients,
\begin{equation*}
I^{\bar p}H_i(c\R P^2;\Z_2)\cong
\begin{cases}
0, & i\geq 2-\bar p(3)=2,\\
I^{\bar p}H_i(\R P^2;\Z_2), & i<2,
\end{cases}
\end{equation*}
and so also $I^{\bar p}H_1(X;\Z_2)\cong \Z_2$ and $I^{\bar p}H_2(X;\Z_2)\cong 0$.

But this violates the universal coefficient theorem, which would predict that $I^{\bar p}H_2(X;\Z_2)$ would equal $\left(I^{\bar p}H_2(X)\otimes \Z_2 \right)\bigoplus \left(I^{\bar p}H_1(X)* \Z_2\right)\cong \Z_2$.

We can see in this example a situation in which a chain that would not be allowable in integer intersection homology becomes allowable in intersection homology with coefficients -- recall from Section \ref{S: background}
that it this effect that is ultimately responsible for the violation of the universal coefficient theorem. Specifically, consider the standard cell decomposition\footnote{We give a description using cells, but this argument could also be made simplicially.} of $\R P^2$ with one cell in each dimension, and let $x$ be  the 1-cell that represents the generator of $H_1(\R P^2)$. Similarly, let $y$ be the 2-cell with $\bd y=2x$. $y$ does not represent an integer homology class because it is not a cycle, but with coefficients in $\Z_2$, $\bd y=0$ and $[y]$ represents
the nontrivial class in $H_2(\R P^2)\cong \Z_2$. This is precisely the term coming from the torsion product in  the universal coefficient theorem in ordinary homology. Now, in $c\R P^2$, consider the $3$-chain $cy$ determined by the cone on $y$. This is \emph{not} an allowable chain with integer coefficients because even though $\dim (cy)\cap X^0=0\leq 3-3+\bar p(3)=0$, we have $\bd(cy)=y-c\bd y=y-c2x$, and this $2$-chain intersects the cone point, which is not allowed with this perversity. However, with $\Z_2$ coefficients, $\bd( cy)=y$, which does not intersection the cone point; thus $cy$ is allowable and kills the cycle $y$.  It is interesting to note that with $\Z$ coefficients $y$ is not even a cycle to begin with!

\paragraph{A ``worse'' violation of universal coefficients.}

In the last example, we saw that intersection homology can violate the universal coefficient theorem when the Goresky-Siegel local torsion condition is violated. More specifically, the expected  torsion product summands did not materialize. In the next example, we see that something more unexpected can happen: the tensor product terms might also vanish. In particular, we will construct spaces that have non-trivial integer intersection homology in their middle dimensions but whose intersection homology with finite coefficients vanishes in the same dimensions.

These examples were motivated initially by applying a  construction that Siegel uses with rational coefficients in \cite{Si83}. We also follow the arguments of Haefliger from  \cite[Section I.5.3]{Bo} for computing the intersection homology of a Thom space.

Consider $p$ copies of $\C P^{2}$ labeled $\C P^2_i$, $i=1,\ldots, p$. In each $\C P_i^2$, let $V_i$ denote an embedded $S^2$ representing the generator of $H_2(\C P^2_i)$. Let $X=\#_{i=1}^p \C P^2_1$. We may assume that the $V_i$ are disjoint within $X$, and we may form the connected sum $V=\#_{i=1}^pV_i\cong S^2$ embedded in $X$. The homology class $[V]\in H_2(\C P^2)$ is equal to $\sum [V_i]$. In particular, since the intersection number $[V_i]\cdot [V_i]=1$ in $\C P^2_i$, we have $[V]\cdot [V]=p$. 

Now let $U$ denote a tubular neighborhood of $V$ in $X$, and let $\hat U$ denote the one point compactification of $U$. This is none other than the Thom space of the normal bundle to $V$ in $X$. As shown in  \cite[Section I.5.3]{Bo}, the intersection homology of a Thom space is easy to compute. In general, if $M$ is a compact $n$-manifold with boundary and $Y=M\cup_{\bd M}c(\bd M)$, a short calculation with the Mayer-Vietoris sequence and the cone formula demonstrates that
\begin{equation*}
I^{\bar p}H_i(Y)\cong
\begin{cases}
H_i(Y), &i>n-\bar p(n)-1,\\
\text{Im: }H_i(M)\to H_i(Y), & i=n-\bar p(n) -1,\\
H_i(M),& i<n-\bar p(n)-1.\\
\end{cases}
\end{equation*}
Roughly speaking, in analogy with the cone formula, chains below a certain dimension are not allowed to intersect the distinguished cone point $v$, and in these dimensions the intersection homology is $H_*(Y-v)\cong H_*(M)$. In high dimensions, any chain is allowed, and the intersection homology is $H_*(Y)\cong H_*(M,\bd M)$. In the transition dimension, cycles cannot intersect $v$, but chains one dimension up can, and so we get the groups $\text{Im: }(H_i(M)\to H_i(Y))\cong \text{Im: }(H_i(M)\to H_i(M,\bd M))$.

If now $M$ is an $r$-disk bundle over a compact $m$-dimensional manifold $B$, then $Y$ is  the associated Thom space, and then we know $H_i(Y)\cong H_i(M, M-B)\cong H_{i-r}(B)$ by the Thom isomorphism theorem. In particular, $\text{Im}(H_i(M)\to H_i(Y))\cong \text{Im}(H_i(B)\to H_i(Y))\cong \text{Im}(H_i(B)\overset{e \cap \cdot}{\to}H_{i-r}(B))$, where $e$ is the euler class of the bundle (see, e.g. \cite[Section VI.12]{BRTG}).

In our case at hand, and using the lower middle perversity $\bar m$ (see Section \ref{S: background}), we therefore have
\begin{equation*}
I^{\bar m}H_i(\hat U)\cong
\begin{cases}
H_{i-2}(S^2), & i>2,\\
\text{Im: } H_2(S^2)\overset{\cap e}{\to} H_{0}(S^2), & i=2,\\
H_i(S^2), & i<2.
\end{cases}
\end{equation*}
So the nonzero groups are $I^{\bar m}H_0(\hat U)=\Z$, $I^{\bar m}H_4(\hat U)=\Z$, and  $I^{\bar m}H_2(\hat U)=p\Z\cong \Z$, since the self-intersection number $[V]\cdot[V]=p$ is equal to the euler number. 

Letting $p$ be a prime, these same calculations hold over the  field  $\Z_p$ except in this case we see that $I^{\bar m}H_0(\hat U;\Z_p)=\Z_p$, $I^{\bar m}H_4(\hat U;\Z_p)=\Z_p$, and,  $I^{\bar m}H_2(\hat U; \Z_p)=p\Z_p=0$. In addition, for a prime $p'\neq p$, we get $I^{\bar m}H_0(\hat U;\Z_{p'})=\Z_{p'}$, $I^{\bar m}H_4(\hat U;\Z_{p'})=p\Z_{p'}\cong \Z_{p'}$, and  $I^{\bar m}H_2(\hat U; \Z_{p'})=\Z_{p'}$.

\begin{remark}\label{R: multip} More generally, if $m$ is a positive composite integer and we perform the above construction with $m$ copies of $\C P^2$, then we will have $I^{\bar m}H_2(\hat U; \Z_p)=0$ for each prime $p$ such that $p\mid m$ but $I^{\bar m}H_2(\hat U; \Z_{p'})=\Z_{p'}$ for each prime $p'$ such that $p'\nmid m$. 
\end{remark}

Obviously there is nothing particularly special here about having found our bundle within a connected sum of $\C P^2$s. In fact, we can perform the same intersection homology computations starting with any $n$-bundle over an $n$-manifold and with an appropriate Euler number. However, our example also illustrates a more general procedure adapted from \cite{Si83} for finding spaces with trivial middle-dimensional $\Z_p$ intersection homology; see Remark \ref{R: more spaces}, below.

\subsection{$K$-Witt spaces that are not $K'$-Witt spaces}\label{S: spaces}

In this section, we construct spaces that are Witt with respect to certain fields but not Witt with respect to others, collectively demonstrating the assertions of Theorem \ref{T: everything} of the Introduction, with the exception of item (1), which is proven in Section \ref{S: bordism}. Recall that item $(1)$ states that whether or not a space is $K$-Witt depends only on the characteristic of $K$; hence in this section we consider only the fields $\Z_p$ and $\Q$.

\paragraph{Low dimensions.} We first dispense with some  low-dimensional considerations, establishing items \eqref{I: 3-4.1}, \eqref{I: 3-4.2}, and \eqref{I: 0-1-2} of Theorem \ref{T: everything}.
We observe immediately that all $0$- and $1$-dimensional pseudomanifolds are manifolds, and hence $K$-Witt for all fields $K$, while  $2$-dimensional pseudomanifolds that are not manifolds can have only codimension $2$ singularities  and so are also $K$-Witt for all $K$. For dimensions $3$ and $4$, we have the following propositions.

\begin{proposition}
Let $X$ be a $3$- or $4$-dimensional $\Z_p$-Witt space. Then $X$ is a $\Q$-Witt space. 
\end{proposition}
\begin{proof}
The only nontrivial even-dimensional links $L$ in $X$ must be $2$-dimensional compact pseudomanifolds. But each of these is the union of a finite number of compact surfaces $S_1,\cdots S_r$, joined along a finite number of points (see \cite{GBF6}). Since intersection homology is invariant under normalizations\footnote{A pseudomanifold is normal if its links are connected, and every pseudomanifold is an image of a finite-to-one cover by a normal manifold, its normalization.} (see \cite[Section 4]{GM1}), $I^{\bar p}H_1(L;\Z_p)\cong H_1(\amalg_i S_i;\Z_p)$, and so the result follows from the universal coefficient theorem for ordinary homology. 
\end{proof}

We also have the following converse:

\begin{proposition}
If $X$ is a $3$- or $4$-dimensional   $\Q$-Witt space, then $X$ is a $\Z_p$ Witt space for any $p\neq 2$. If $X$ is also $\Q$-orientable, then it is also a $\Z_2$-Witt space. However, there are non-$\Q$-orientable $3$- and $4$-dimensional   $\Q$-Witt spaces that are not $\Z_2$-Witt spaces.  
\end{proposition}
\begin{proof}

The  proposition is true for $p\neq 2$, since if $X$ is a $3$- or $4$-dimensional $\Q$-Witt space, then the two-dimensional links must consist of $S^2$s and $\R P^2$s glued along points, and these links will then have trivial  $\Z_p$ intersection homology in degree one, as well. 

For $p=2$, if a $\Q$-Witt space has a link that does involve at least one $\R P^2$, the space will not be $\Z_2$-Witt. However, we claim such a space will not be $\Q$-orientable either. To see this, note that any such pseudomanifold must have distinguished neighborhoods of the form $cL$ or  $\R^1\times cL$ and for which there is a map $\R P^2\to L$ that is injective off of finitely many points. Any embedded curve representing a generator of $\pi_1(\R P^2)$ can be homotoped by a small homotopy to an embedded curve $\gamma$ in $L$ whose neighborhood in the distinguished neighborhood is homeomorphic to the product of a M\"obius band with $\R^1$ or $\R^2$. Thus tracing around $\gamma$ in $X$ must reverse orientation. 

For  examples of non-orientable $3$- and $4$-dimensional PL  $\Q$-Witt spaces that are not $\Z_2$-Witt, we can take the suspension and double suspension of $\R P^2$.
\end{proof}

\begin{remark}We note that the non-$\Q$-orientable $\Q$-Witt non-$\Z_2$-Witt spaces also technically violate Siegel's definition of a $\Q$-Witt space in \cite{Si83}, since it implicitly uses the original Goresky-MacPherson definition of a pseudomanifold from \cite{GM1}, and this definition includes an orientability condition. In any event, there will be no rational perfect pairing on middle intersection homology, and this is what we like Witt spaces for. 
\end{remark}

\paragraph{Spaces that are $\mathbf{\Z_p}$-Witt but not $\mathbf{\Q}$-Witt or $\mathbf{\Z_{p'}}$-Witt for $\mathbf{p'\neq p}$.} 
We turn to demonstrating item \eqref{I: item} of Theorem \ref{T: everything} by next
constructing, for $n>4$,  an $n$-dimensional $\Z_p$-Witt space, $p$-prime,  that is neither $\Q$-Witt nor $\Z_{p'}$-Witt for any prime $p'\neq p$. From the preceding section, there exists a $4$-dimensional pseudomanifold $\hat U$ whose $I^{\bar m}H_2(\hat U;\Z_p)$ vanishes but such that neither $I^{\bar m}H_2(\hat U;\Q)$ nor $I^{\bar m}H_2(\hat U;\Z_{p'})$ vanishes for prime $p'\neq p$. It follows immediately that the suspension $S\hat U$ is $\Z_p$-Witt but not $\Q$-Witt and not  $\Z_{p'}$-Witt for any $p'\neq p$. By taking products with manifolds, $M\times S\hat U$, we obtain compact $\Z_p$-Witt spaces that are not Witt for fields of any other characteristic in all dimensions $\geq 5$. 

However, to make these examples a bit more robust, we would like to find some $4k$-dimensional spaces with these Witt properties and for which we can see directly that there is a nontrivial nonsingular middle intersection pairing over $\Z_p$ but not over $\Q$ or $\Z_{p'}$ for $p'\neq p$. For this, we use a slightly more elaborate starting point.

Consider a bundle of $2$-planes over the torus $T^2\cong S^1\times S^1$ and with euler number $e=p$. These can be found, for example, by providing $T^2$ with a complex structure and then forming the complex line bundle associated to a divisor $p[x]$ for $x\in T^2$; see, e.g. \cite{HUY}. Let $Y$ be the associated Thom space. Then, by our computations above, we have

\begin{align*}
I^{\bar m}H_i(Y;\Q)&\cong
\begin{cases}
H_{i-2}(T^2;\Q), & i>2,\\
\Q\cong \text{Im: } H_2(T^2;\Q)\overset{\cap e}{\to} H_{0}(T^2;\Q), & i=2,\\
H_i(T^2;\Q), & i<2,
\end{cases}\\
&\cong
\begin{cases}
\Q, & i=4,\\
\Q\oplus \Q, &  i=3,\\
\Q, &i=2,\\
\Q\oplus \Q, & i=1,\\
\Q, & i=0,
\end{cases}
\end{align*}
and with coefficients in $\Z_{p'}$, we obtain the same results with each $\Q$ replaced by $\Z_{p'}$. Meanwhile,

\begin{align*}
I^{\bar m}H_i(Y;\Z_p)&\cong
\begin{cases}
H_{i-2}(T^2;\Z_p), & i>2,\\
0\cong \text{Im: } H_2(T^2;\Z_p)\overset{\cap e}{\to} H_{0}(T^2;\Z_p), & i=2,\\
H_i(T^2;\Z_p), & i<2,
\end{cases}\\
&\cong
\begin{cases}
\Z_p, & i=4,\\
\Z_p\oplus \Z_p, &  i=3,\\
0, &i=2,\\
\Z_p\oplus Z_p, & i=1,\\
\Z_p, & i=0.
\end{cases}
\end{align*}

If $\alpha, \beta$ denote cycles  generating $H_1(S^1\times S^1;\Z_p)$, then $\alpha$ and $\beta$ also generate $I^{\bar m}H_1(Y;\Z_p)$, while $I^{\bar m}H_3(Y;\Z_p)$ is generated by
the restrictions of the Thom space over $\alpha$ and $\beta$, say $\hat \alpha, \hat \beta$ (each of which is homeomorphic as a space to the one point compactification of $S^1\times \R^2$ since the bundle is trivial over the complement of the divisor).

Now let $X=S^1\times S^2\times SY$, where $SY$ is the suspension of $Y$. The only singular stratum is $S^1\times S^2\times \{N,S\}$, where $\{N,S\}$ represents the north and south poles of the suspension. This is a codimension $5$ stratum of an $8$-dimensional pseudomanifold, and the link of the stratum is $Y$. Since $Y$ has vanishing middle dimensional middle perversity intersection homology over $\Z_p$, $X$ is a $\Z_p$-Witt space, but the middle intersection homology fails to vanish over $\Q$ so that $X$ is not a $\Q$-Witt space.

Using the formula for the intersection homology of a suspension (see Section \ref{S: background}),  together with the K\"unneth theorem, which holds holds for intersection homology when one term is a manifold (see \cite{Ki}), we see that 

$$I^{\bar m}H_4(X;\Z_p)\cong \Z_p\oplus \Z_p\oplus \Z_p\oplus \Z_p.$$

If $*$ denotes a basepoint in $S^1\times S^2$, the generators are  
$*\times S\hat \alpha$, $*\times S\hat \beta$, $S^1\times S^2\times \alpha$ and $S^1\times S^2\times \beta$. If the intersection number $\alpha\cdot \beta=1$, then the intersection matrix with respect to this basis is 

\begin{equation*}
\begin{pmatrix}
0 & 0 & 0 &  1\\
0 & 0 & -1 & 0\\
0 &  -1 &0&0\\
1 &0 &0&0 
\end{pmatrix},
\end{equation*}
assuming that $SX$ is thought of as $I\times X/\sim$ (as opposed to $Z\times I/\sim$).

On the other hand, $$I^{\bar m}H_4(X;\Q)\cong \Q\oplus \Q\oplus \Q\oplus \Q \oplus \Q,$$
where the first four summands are generated as before and the additional summand is generated by $z=*\times S^2$ times the generator of $I^{\bar m}H_2(X;\Q)$, which can be represented by $T^2$. Its intersection with each generator of $I^{\bar m}H_4(X;\Q)$, including itself, is $0$. For the intersections with $*\times S\hat \alpha$, $*\times S\hat \beta$, and itself, this can be seen by pushing it off the basepoint in the $S^1$ direction. For the intersections with $S^1\times S^2\times \alpha$ and $S^1\times S^2\times \beta$, we can push $z$ off in the direction of the suspension. Thus we obtain a degenerate intersection pairing. The dual of $z$ lives, of course, in $I^{\bar n}H_4(X;\Q)\cong \Q^5$, which is generated by our earlier four generators and $S^1\times *\times ST^2$. With the above convention for suspensions, we have an intersection number $(*\times S^2\times T^2)\cdot (S^1\times *\times ST^2)=p$. For $\Z_{p'}$, $p\neq p'$, the computations are the same, replacing all $\Q$s by $\Z_{p'}$s. 

Similar examples may be obtained easily in higher dimensions. For example, in dimensions $4k$, $k>2$, we can take the product of $X$ from the previous example with $k-2$ copies of $\C P^2$. The new space will be Witt (or non-Witt) for exactly the same fields as for $X$, the middle-dimensional pairing over $\Z_p$ remains nontrivial (and nonsingular), and, if $V_i$ represents the sphere in the $i$th copy of $\C P^2$ generating the homology in degree $2$, then $*\times S^2\times T^2\times \prod_{i=1}^{k-2}V_i$ represents a nontrivial $\bar m$-allowable class over $\Q$ and $\Z_{p'}$ whose dual $S^1\times *\times ST^2\times \prod_{i=1}^{k-2}V_i$ is $\bar n$-allowable but not $\bar m$-allowable. 

Furthermore, applying Remark \ref{R: multip} from above, if we carry through the above procedure for a bundle with euler number $m$, a composite instead of a prime, then we obtain spaces that are $\Z_p$-Witt for all primes $p$ such that $p\mid m$ but not $\Q$-Witt nor $\Z_p$-Witt when $p\nmid m$.

\paragraph{$\mathbf{\Q}$-Witt, but not $\mathbf{\Z_p}$-Witt for some $\mathbf{p}$.}

We now look for spaces that are $\Q$-Witt but that fail to be $\Z_p$-Witt for a single prime or a collection of primes. This corresponds to item \eqref{I: Qwitt} of Theorem \ref{T: everything}. In general, obtaining such spaces is easy; for example, take any even-dimensional closed manifold whose middle homology is all torsion and suspend as many times as desired. However, we would once again like to verify that these actually exhibit the correct existence or lack of middle-dimensional pairings (at least for spaces of dimension $4k$, $k>1$). It turns out that we can do even this without having to resort to constructions quite as specialized as those in the last section; in particular, we can start with manifolds and introduce a singularity with just a single suspension. 

To start off, fix a prime $p$, and let $L$ be  a $3$-dimensional lens space with $H_1(L)\cong \Z_{p}$ (see, e.g. \cite[Section 40]{MK}). Then we have $H_0(L;\Q)\cong H_3(L;\Q)\cong \Q$, $H_1(L;\Q)=H_2(L;\Q)=0$, and the same formulas replacing $\Q$ everywhere by $\Z_{p'}$ for $p'\neq p$. However, for coefficients in $\Z_p$, $H_i(L;\Z_p)\cong \Z_p$ for $0\leq i\leq 3$. This follows from the ordinary integer homology of the lens space and the universal coefficient theorem.

If we now let $J=L\times S^1$, then $J$ is a compact orientable $4$-manifold with $H_2(J;\Z_p)\cong \Z_p\oplus \Z_p$ and $H_2(J;\Q)=H_2(J;\Z_{p'})=0$ for $p'\neq p$. Thus the suspension $SJ$ is $\Q$-Witt and $\Z_{p'}$-Witt for all $p'\neq p$, but it is not $\Z_p$-Witt. By taking products with manifolds, we obtain spaces of all dimensions $\geq 5$ with these properties. 

Computing with the ordinary K\"unneth Theorem and the suspension formula for intersection homology (see Section \ref{S: background}, above), we see that

\begin{align*}
I^{\bar m}H_i(SJ;\Q)&\cong
\begin{cases}
\Q, & i=5,\\
\Q, & i=4,\\
0, &  i=3,\\
0, &i=2,\\
\Q, & i=1,\\
\Q, & i=0,
\end{cases}
\end{align*}
and similarly with all $\Q$s replaced by $\Z_{p'}$. On the other hand,
 
\begin{align*}
I^{\bar m}H_i(SJ;\Z_p)&\cong
\begin{cases}
\Z_p, & i=5,\\
\Z_p\oplus \Z_p, & i=4,\\
0, &  i=3,\\
\Z_p\oplus \Z_p, &i=2,\\
\Z_p\oplus \Z_p, & i=1,\\
\Z_p, & i=0.
\end{cases}
\end{align*}

Now, consider $X=SJ\times S^1\times S^2$, which is an $8$-dimensional pseudomanifold. We have $I^{\bar m}H_4(X;\Q)\cong \Q\oplus \Q$, generated by the suspension $S(L\times *_{S^1})\times *_{S^1\times S^2}$ and by $(*_{L}\times S^1)\times S^1\times S^2$. These cycles are readily checked to be dual to each other. The same is true replacing $\Q$ with $\Z_{p'}$.

On the other hand, let $d_i$ be the $i$-cell in the standard decomposition of the lens space with one cell in each dimension (see \cite{MK}). Then $I^{\bar m}H_4(X;\Z_p)\cong \Z_p^6$. The generators are: 
\begin{diagram}
S(L\times *_{S^1})\times *_{S^1\times S^2}&\qquad\qquad& S(d_2\times S^1)\times *_{S^1\times S^2}\\
 (d_2\times *_{S^1} )\times *_{S^1}\times S^2 &&(d_1\times S^1) \times *_{S^1}\times S^2\\
(d_1\times *_{S^1})\times S^1\times S^2&& (d_0\times S^1)\times S^1\times S^2.
\end{diagram} 
 
  It is the middle row of  generators that do not have appropriate duals in  $I^{\bar m}H_4(X;\Z_p)$. Their duals should be, respectively, $ S(d_1\times S^1)\times S^1\times *_{S^2}$ and  $S(d_2\times *_{S^1}) \times S^1\times *_{S^2}$, which, of course, are generators of $I^{\bar n}H_4(X;\Z_p)$. One readily checks geometrically that the intersection numbers of
 $ (d_2\times *_{S^1} )\times *_{S^1}\times S^2$ and $(d_1\times S^1) \times *_{S^1}\times S^2$
are  $0$ with all other generators of $I^{\bar m}H_4(X;\Z_p)$ - with the first two and the middle two by pushing off in the $*\times S^1\times *$ direction and with the last two by pushing up or down in the direction of the suspension. 

From here, we may once again obtain examples in all dimensions $4k$, $k>2$, by taking products with $\C P^2$s. Also, by taking connected sums with spaces constructed in the exact same way but for different primes in a set $P=\{p_i\}$, we obtain spaces that are $\Q$-Witt and $\Z_{p}$-Witt for any $p\notin P$ but that are not $\Z_{p}$-Witt for any $p\in P$.

\paragraph{Remaining questions.} 
We leave the following as open questions:

\begin{question}
Are there spaces that are $\Q$-Witt but that are not $\Z_{p}$-Witt for an \emph{infinite} set of primes? Are there spaces that are $\Q$-Witt but that are not $\Z_p$-Witt for all but a finite set of primes?
\end{question}

There can be no such compact example with all links compact \emph{manifolds}, since the middle dimensional homology groups would be finitely generated and thus not capable of carrying  infinite different types of torsion. 

\begin{question}
Are there spaces that are not $\Q$-Witt but that are $\Z_{p}$-Witt for an \emph{infinite} set of primes? Are there spaces that are not $\Q$-Witt but that are $\Z_p$-Witt for all but a finite set of primes?
\end{question}

\section{$K$-Witt bordism groups}\label{S: bordism}

In this section, we discuss the adaptation of Siegel's theorem on $\Q$-Witt bordism to other coefficient fields. In \cite{Si83}, Siegel notes that, as a consequence of the Poincar\'e duality on $\Q$-Witt spaces, for each $k>0$ there is a well-defined homomorphism from the Witt bordism group $\Omega_{4k}^{\Q-\text{Witt}}$ of compact $4k$-dimensional $\Q$-Witt spaces to the Witt group $W(\Q)$ of nondegenerate symmetric $\Q$-bilinear forms, given by taking a $\Q$-Witt space to its middle dimensional middle-perversity intersection form. One of the principal results of \cite{Si83} is that this homomorphism is, in fact, an isomorphism and that these bordism groups are $0$ in all other dimensions except for $k=0$, which has $\Omega_{0}^{\Q-\text{Witt}}\cong \Z$. It then follows from work of Sullivan that, as a homology theory, $\Q$-Witt bordism, $\Omega^{\Q-\text{Witt}}(\cdot)$, is equivalent to $KO[1/2](\cdot)$.  In this section, we extend Siegel's results by computing the $K$-Witt bordism groups for an arbitrary field $K$. We prove Theorems \ref{T: witt1} and \ref{T: witt2}, stated in the introduction. The reader can find more background on Witt groups in \cite{MH, LAM}.

In Subsection \ref{S: prelim}, we provide the basic definitions and some preliminary observations. In Subsection \ref{S: reduction}, we show that the Witt bordism groups (in fact the property of being a $K$-Witt space) depends only on the characteristic of $K$. 
In Subsection \ref{S: finite}, we prove the analogue of Siegel's theorem, $\Omega_{4k}^{\Z_p-\text{Witt}}\cong W(\Z_p)$ for $k>0$. In Subsection \ref{S: Witt maps}, we examine more closely the map $\Omega_{4k}^{\F_q-\text{Witt}}\to W(\F_q)$ for the finite fields $\F_q$. Finally, in Subsection \ref{S: homology}, we show that, as a homology theory, $\Z_p$-Witt bordism splits into a sum of (shifted) ordinary homology groups with coefficients.

\subsection{Preliminaries}\label{S: prelim}

Let $K$ be a field. A PL space $X$ is a \emph{$K$-Witt space with boundary} if $\bd X$ and $X-\bd X$ are PL pseudomanifolds that satisfy the $K$-Witt condition and $\bd X$ has a collar in $X$. This collared boundary requirement is a more restrictive condition than allowing codimension $1$ strata in general (which are referred to as ``pseudoboundaries'' in \cite{GBF11}).
We let $\Omega_n^{K-\text{Witt}}$ denote the group of bordism classes of $n$-dimensional $K$-Witt spaces, in which $X$ is trivial if $X$ is the boundary of an $n+1$ dimensional $K$-Witt space with boundary. See \cite{Si83, GoBo} for more details in the $\Q$-Witt case.

The main invariant of the Witt bordism groups comes from the middle-dimensional intersection pairings. For a $2k$-dimensional compact oriented $K$-Witt space, there is a nondegenerate $(-1)^k$-symmetric intersection pairing $I^{\bar m}H_k(X;K)\otimes I^{\bar m}H_k(X;K)\to  K$; see \cite{GM1, GBF18} for more on the intersection pairing. 
For $n\equiv 0\mod 4$,  the resulting homomorphism $w=w_K:\Omega_n^{K-\text{Witt}}\to W(K)$ is well-defined. The proof of this in the $K$ case is exactly the same as that for $\Q$-Witt spaces  given in \cite[Theorem 2.1]{Si83}, which itself uses intersection homology Poincar\'e duality to show that the intersection form of a boundary must have a self-annihilating subspace of half the dimension of $I^{\bar m}H_{k}(X;K)$. The basic idea is exactly the same as the proof of signature invariance under manifold bordism. Any pairing with a self-annihilating subspace of half the dimension is trivial in the Witt group; see \cite{MH}.

\subsection{Reduction to prime fields}\label{S: reduction}

First, we reduce the problem of computing $\Omega_n^{K-\text{Witt}}$ to the special cases where $K=\Z_p$ or $\Q$ by showing that whether or not $X$ is a $K$-Witt space is determined entirely by the characteristic of $K$. For ease of treating all cases simultaneously, we define $\Z_0:=\Q$. We state the key results and then provide the proofs at the end of the subsection.

\begin{lemma}\label{L: K=p}
Let $X$ be a PL stratified pseudomanifold. Suppose $K$ is a field of characteristic $p$ (which may be $0$). Then, (letting $\Z_0=\Q$), there is a chain isomorphism $$I^{\bar p}C_*(X;\Z_p)\otimes_{\Z_p}K\to I^{\bar p}C_*(X;K).$$ 
\end{lemma}

\begin{corollary}\label{C: K=p}
Let $X$ be a PL stratified pseudomanifold and $K$ a field of characteristic $p$, possibly with $p=0$. Then $X$ is $K$-Witt if and only if $X$ is $\Z_p$-Witt (taking $\Z_0=\Q$). 
\end{corollary}
\begin{corollary}\label{C: 2}
If $K$ is of characteristic $p$, $\Omega_n^{K-\text{Witt}}=\Omega_n^{\Z_p-\text{Witt}}$. In particular, if $K$ has characteristic $0$,  $\Omega_n^{K-\text{Witt}}=\Omega_n^{\Q-\text{Witt}}$.
\end{corollary}

Corollary \ref{C: 2} follows immediately from Corollary \ref{C: K=p}.

To see how this reduction relates to the Witt group invariants, note that for any inclusion of fields $\Z_p \into K$ (or $\Q\into K$ if $p=0$), we obtain a commutative diagram

\begin{equation}\label{E: commute}
\begin{CD}
\Omega_{4k}^{\Z_p-\text{Witt}}&@>w_{\Z_p}>>&W(Z_p)\\
@V= VV&&@VVV\\
\Omega_{4k}^{K-\text{Witt}}&@>w_{K}>>&W(K).
\end{CD}
\end{equation}
Here the righthand vertical map is induced by the field homomorphism $\Z_p \into K$.
To see that this diagram commutes, first note that by Lemma \ref{L: K=p},  $I^{\bar m}H_{2k}(X;K)$ is generated over $K$ by elements of the form $[\xi]\otimes 1$, where  $\xi\in I^{\bar m}H_{2k}(X;\Z_p)$. 
So,  if $X^{4k}$ is $K$-Witt (and hence $\Z_p$-Witt) and we  choose a basis for $I^{\bar m}H_{2k}(X;\Z_p)$, then the resulting intersection pairing matrix for the dual $K$ pairing on $I^{\bar m}H_{2k}(X;K)$ is identical to the intersection pairing matrix for the $\Z_p$ pairing on $I^{\bar m}H_{2k}(X;\Z_p)$. Thus the same matrix with entries in $\Z_p$ represents both $w_{\Z_p}(X)\in W(\Z_p)$ and $w_{K}(X)\in W(K)$. This is consistent with the map $W(\Z_p)\to W(K)$ induced by inclusion.

It follows that $w_K$ is determined entirely by $w_{\Z_p}$, which will be studied in the next subsection for $p\neq 0$.

 We return now to the deferred proofs.

\begin{proof}[Proof of Lemma \ref{L: K=p}]
Define $\Phi:I^{\bar p}C_*(X;\Z_p)\otimes_{\Z_p}K\to I^{\bar p}C_*(X;K)$ by $\Phi(\sum_i \xi_i\otimes k_i)=\sum_i k_i\xi_i$, where each $\xi_i\in I^{\bar p}C_*(X;\Z_p)$ and $k_i\in K$. Note that $k_i\xi_i$ makes sense as each $\xi_i$ equals $\sum m_j\sigma_j$ for $m_j\in\Z_p$ and $\sigma_j$ a simplex of some triangulation of $X$, and $k_im_j$ makes sense as an element of $K$. It is easy to check that $\Phi$ is a chain map and well-defined.

Perhaps the simplest way to check that $\Phi$ is injective is to consider the following commutative diagram:
\begin{diagram}
I^{\bar p}C_*(X;\Z_p)\otimes_{\Z_p}K&\rTo^\Phi &I^{\bar p}C_*(X;K)\\
\dInto && \dInto\\
C_*(X;\Z_p)\otimes_{\Z_p}K&\rEquals & C_*(X;K).
\end{diagram}
Note that it is clear from the definitions that $I^{\bar p}C_*(X;G)\to C_*(X;G)$ is always an inclusion. Since the tensor product over a field is left exact, the vertical arrows are inclusions, and the bottom map is a standard isomorphism. It follows from the diagram that $\Phi$ is injective.

For surjectivity, let $\xi=\sum k_i\sigma_i \in I^{\bar p}C_*(X;K)$. Note that the sum is finite, since each element of $C_*(X)$ lives in a fixed triangulation. Consider the $\Z_p$ vector subspace $V$ of $K$ spanned by the $k_i$. Let $\{x_j\}_{j=1}^N$ be a basis for $V$. Then each $k_i=\sum n_jx_j$, $n_j\in \Z_p$. Using this, we can rewrite $\xi$ in the form $\sum x_j\xi_j$, where $\xi_j\in C_*(X;\Z_p)$. We claim that each $\xi_j\in I^{\bar p}C_*(X;\Z_p)$, from which it will follow that $\xi=\Phi(\sum \xi_j\otimes x_j)$. It is clear that each simplex $\sigma$ appearing in each $\xi_j$ must be allowable since each occurs with non-zero coefficient in $\xi$. The point is to show that each $\bd \xi_j$ is allowable, which is not immediately clear. However, suppose that $\tau$ is a simplex that appears with nonzero ($\Z_p$-)coefficient in $\bd \xi_j$ for some $j$. The total coefficient of $\tau$ in $\bd \xi$ must have the form $\sum_j x_jm_j$, where $m_j$ is the coefficient of $\tau$ in $\bd \xi_j$. Since $\tau$ appears nontrivially in $\bd \xi_j$ for some $j$, some $m_j\not\equiv 0\mod p$, and so $\sum_j x_jm_j\neq 0$, as the $\xi_j$ are linearly independent as a vector space basis. Thus $\tau$ appears nontrivially in $\bd \xi$, and hence must be allowable, because $\xi$ is an allowable chain.
\end{proof}

\begin{remark}
This lemma can be shown more generally over  topological pseudomanifolds using the sheaf approach to intersection homology. We provide a PL chain level proof, more in keeping with the spirit of the current paper.
\end{remark}

\begin{proof}[Proof of Corollary \ref{C: K=p}]
From the lemma and the algebraic universal coefficient theorem for fields, for a compact PL pseudomanifold, $I^{\bar p}H_*(X;K) \cong I^{\bar p}H_*(X;\Z_p)\otimes_{\Z_p}K$. So, for a compact link $L$, $I^{\bar p}H_i(L;K)$ vanishes if and only if $I^{\bar p}H_i(L;\Z_p)$ vanishes. The corollary follows.   
\end{proof}

\subsection{Witt bordism over $\Z_p$}\label{S: finite}

In this section we compute $\Omega^{\Z_p-\text{Witt}}_n$. For a further discussion of the case of more general finite fields, see Section \ref{S: Witt maps} below.

\begin{theorem}\label{T: F-Witt}
Let $p\neq 0$ be a prime. 
\begin{enumerate}
\item $\Omega_0^{\Z_p-\text{Witt}}\cong \Z$. 

\item For $n\not \equiv 0\mod 4$, $\Omega_n^{\Z_p-\text{Witt}}=0$. 

\item For $n\equiv 0\mod 4$, $n>0$, the homomorphism $w:\Omega_n^{\Z_p-\text{Witt}}\to W(\Z_p)$ that assigns to $X^n$ the intersection form on $ I^{\bar m}H_{n/2}(X;\Z_p)$ is an isomorphism.
\end{enumerate}
\end{theorem}  

We can observe immediately that all $0$- and $1$-dimensional $\Z_p$-Witt spaces are, of course, manifolds, so that $\Omega_0^{\Z_p-\text{Witt}}\cong \Z$, as for manifolds and $\Q$-Witt spaces. Furthermore, as for $\Q$-Witt spaces, $\Omega_{2k+1}^{\Z_p-\text{Witt}}=0$, since if $X$ is an odd-dimensional $\Z_p$-Witt space, then the closed cone $\bar cX$ is also $\Z_p$-Witt, as the new stratum consisting of the cone vertex has even codimension. 

This leaves the Witt spaces of positive even dimension.

We must next show the following:

\begin{proposition}\label{P: p}  $w:\Omega_n^{\Z_p-\text{Witt}}\to W(\Z_p)$ is an isomorphism for all primes $p$ and $n=4k>0$, and $\Omega_n^{\Z_p-\text{Witt}}=0$ for $n\equiv 2\mod 4$.
\end{proposition}

\begin{proof}[Proof of Proposition \ref{P: p}]
The proof  is mostly analogous to that of Siegel's theorem for $\Q$-Witt bordism. The main difference is the proof of surjectivity for $0\neq q\equiv 0\mod 4$. We will demonstrate this surjectivity and then discuss the one significant change from Siegel's proof of injectivity that we must make  in the $\Z_p$ situation. 
The proof that $\Omega_q^{\Z_p-\text{Witt}}=0$ for $q\equiv 2\mod 4$ is included with the proof of injectivity of $w$ for $0<q\equiv 0\mod 4$.

The surjectivity of  $w:\Omega_{4k}^{\Z_p-\text{Witt}}\to W(\Z_p)$ is contained in the following proposition.

\begin{proposition}\label{P: surj}
Given any element of $x\in W(\Z_p)$ and any integer $k>0$, there exists a $4k$-dimensional $\Z_p$-Witt space $X$ whose intersection pairing on $I^{\bar m}H_{2k}(X;\Z_p)$ represents $x$. 
\end{proposition}
\begin{proof}
It follows from the theory of Witt rings (see \cite[Chapter 2]{LAM}) that for a finite field $F$, the Witt ring $W(F)$ is isomorphic to a semi-direct product $Q(F)=\Z_2 \ltimes \dot F/\dot F^2$, where $\dot F$ is the 
multiplicative group $F-\{0\}$ and $\dot F/\dot F^2\cong \Z_2$. The group operation in $Q(F)$ is given by $(e,f)\cdot (e',f')=(e+e',(-1)^{ee'}ff')$, and the isomorphism $W(F)\to Q(F)$ is given by the pair of operators $(\dim_0,d_{\pm})$. Here $\dim_0$ takes a representative symmetric bilinear form to its dimension $\mod 2$ and $d_{\pm}$ takes a representative form of dimension $n$ to $(-1)^{n(n-1)/2}$ times the determinant of the matrix representing the form.

When $F=\F_q$ with $q\equiv 3 \mod 4$, then $W(F)\cong Q(F)\cong \Z_4$. The generator is $(1,1)$ and $Q(F)=\{(0,1), (1,1), (0,-1),(1,0)\}$ (note that $-1$ is not a square when $q\equiv 3 \mod 4$). In terms of forms, $W(F)$ is therefore generated by $\langle 1\rangle$.

If $q\equiv 1 \mod 4$, then $Q(F)$ splits as $\Z_2\times \Z_2$. If $s\in \F_q$ is not a square, then the elements of $Q(F)$ are represented by $\{(0,1), (1,1), (1,s), (0,s)\}$, generated in $W(F)$ by $\langle 1\rangle$ and $\langle s\rangle$.

If $q$ is even, then $W(\F_q)\cong \Z_2$ \cite{MH}.

So when $p\equiv 3 \mod 4$, it is easy to find manifolds of dimension $4k$, $k\geq 1$ whose middle dimensional pairings represent any element of the Witt group. Specifically, the pairing on $\C P^{2k}$ is the generator of $W(\Z_p)$ and taking connected sums of $\C P^k$s yields representatives of any element of $W(\Z_p)$.

For $p\equiv 1 \mod 4$, $\C P^{2k}$ again yields the generator $\langle 1\rangle$ of the Witt group $W(\Z_p)$. To obtain the generator $\langle s\rangle$, we can proceed with Thom spaces as in Section \ref{S: worse}. If $s>0$ represents  a non-square unit in $\Z_p$, let $U$ be the normal bundle of the connected sum of $s$ copies of the generator $[V_i]\in H_{2k}(\C P^{2k})$, represented by embedded spheres $S_i^{2k}$, in the connected sum of $s$ copies of $\C P^{2k}$.\footnote{Note that here it is crucial that we work in $\Z_p$ and not some $\F_{p^n}$ so that $s$ can be represented by an integer.} Let $\hat U$ be the associated Thom space. Then, as in the computations in Section \ref{S: worse}, $I^{\bar m}H_{2k}(\hat U;\Z_p)\cong \text{Im: } (H_{2k}(S^{2n};\Z_p)\cong \Z_p \overset{\cap e}{\to} H_0(S^{2k};\Z_p)\cong \Z_p)$. Since the euler number of the bundle is the unit $s$, $I^{\bar m}H_{2k}(\hat U;\Z_p)\cong \Z_p$ generated by $[V]=\sum [V_i]=[S^{2n}]$. Furthermore, $[V]\cdot [V]=s$. Thus, since the only singularity of $\hat U$ has codimension $4k$,  $\hat U$ is a Witt space with intersection form $\langle s\rangle$. So we have constructed spaces representing both generators of $W(\Z_p)$. 

For $p=2$, $W(\Z_p)$ is generated simply by $\langle 1\rangle$, and we can use again any $\C P^{2k}$ as a geometric realization.
\end{proof}

Turning to the injectivity of  $w:\Omega_{4k}^{\Z_p-\text{Witt}}\to W(\Z_p)$, as well as the fact that $\Omega_{n}^{\Z_p-\text{Witt}}=0$ for $n\equiv 2\mod 4$, we note again that the proof is nearly identical to that of the $\Q$-Witt case, though  some care must be taken, primarily with the use of geometric cycles (see the discussion at the end of Section \ref{S: Witt maps}
 for more elaboration on what can go wrong over fields more general than $\Z_p$). We discuss this issue and refer the reader to \cite{Si83} for the remainder of the proof.

Siegel's proof of injectivity over $\Q$ begins by supposing we have a $\Q$-Witt space  $X^{2k}$ with an isotropic element $[z]\in I^{\bar m}H_k(X;\Q)$, i.e. $[z]\cdot [z]= 0$. Siegel then finds a representative of $[z]$ by an \emph{irreducible cycle} $z$, meaning that  $H_k(|z|;\Z)\cong \Z$ and such that the generator of this homology group has coefficient $\pm 1$ on every $k$ simplex of $|z|$ in some triangulation of $|z|$. 
The key point is that the support $|z|$ of $z$ should have infinite cyclic $k$th homology, and it should be generated by $|z|$, itself, considered as the cycle represented by its fundamental class. 

In the $\Z_p$ case, the construction of $z$ should be altered slightly. 
Given any $2k$-dimensional Witt space, any class $[z]\in I^{\bar m}H_k(X;\Z_p)$ will satisfy $[z]\cdot [z]=0$ if $2k\equiv 2\mod 4$, while if $2k\equiv 0\mod 4$ and $w(X)=0\in W(\Z_p)$, such a cycle certainly exists.  We   construct a ``$\Z_p$-irreducible'' representative cycle $z$ for $[z]$ by slightly modifying Siegel's construction to obtain a $z$ such that $H_k(|z|;\Z_p)\cong \Z_p$, generated by a ``fundamental class'' of $|z|$. The quotation marks indicate that this is not quite the right language since $|z|$ might have a codimension one singularity - however since $z$ is a cycle, these singularities will cancel when thinking of $z$ as a chain, and so the idea of a fundamental class makes some sense.  

Briefly, we make $z$ irreducible  as follows: Choose an arbitrary representative $y$ for  $[z]$ in some triangulation of $X$. Then $y=\sum n_i\sigma_i$, where $n_i\in \Z_p$ and each $\sigma_i$ is a unique oriented $k$-simplex. Choose $m_i\equiv n_i\mod p$ such that $0\leq m_i<p$, and abusing notation, let $y=\sum m_i\sigma_i$. Note that the interior of each $k$-simplex of $y$ lies in $X-\Sigma$ due to the allowability conditions for an intersection chain. Now, for each $\sigma_i$ such that $m_i\neq 0$, we may use a relative stratified general position argument following McCrory \cite{Mc78} (see also \cite{GBF18}) to PL isotope the interiors of the  $m_i$ copies of $\sigma_i$, rel boundary, into stratified  general position with respect to each other, and in such a way that none of the new $m_i$ copies of $\sigma_i$ intersect any of the $m_j$ copies of $\sigma_j$ similarly created, except along boundaries. This is more or less an alternative description of Step 1 of the proof of Siegel's \cite[Lemma 2.2]{Si83}, except that Siegel separates by isotopy entire open $j$-strata of $|y|$. Arguments along the lines of the stratified homotopy invariance of intersection homology (see \cite{GBF3}) show that this new chain, $\bar y$, also represents $[z]$, and clearly $\bd \bar y=\bd y=0\in C_{k-1}(X;\Z_p)$. Furthermore, $y$ has the form $y=\sum \sigma_i$, where the sum is taken over those oriented simplices in the support of $y$. To form $z$, one then connects all of the $k$-simplices of $y$ by orientation respecting pipes; see \cite[page 1087]{Si83} for more details. Then one also has $z$ of the form $z=\sum \sigma_i$ (with different $\sigma$s from $\bar y$), and clearly $z$ generates $H_k(|z|;\Z_p)\cong \Z_p$.

This $z$ can then be used in the remainder of a $\Z_p$ analogue of Siegel's injectivity proof. 

\end{proof}

This completes the proof of Theorem \ref{T: F-Witt}.\hfill\qedsymbol

\begin{remark}\label{R: more spaces}
Incidentally, the process of the injectivity proof can be used to find many more examples of $\Z_p$-Witt spaces with vanishing middle-dimensional $I^{\bar m}H_k$ since, by copying the proof of Siegel's \cite[Proposition III.3.1]{Si83} exactly, $I^{\bar m}H_k(\hat U;\Z_p)=0$, where $U$ is a regular neighborhood of our irreducible $z$ (still assuming $[z]\cdot[z]=0$), and $\hat U=U\cup \bar c(\bd U)$. 
\end{remark}

\subsection{Witt bordism over finite fields}\label{S: Witt maps}

In this section, we observe the ramifications of the results of the previous section to $\Omega_n^{\F_q-\text{Witt}}$ for the finite field $\F_q$. We also provide an  illustrative example.

\begin{proposition}
For $n\equiv 0\mod 4$, $n>0$, the homomorphism $w:\Omega_n^{\F_q-\text{Witt}}\to W(\F_q)$ that assigns to $X^n$ the intersection form on $ I^{\bar m}H_{n/2}(X;\F_p)$ is an isomorphism, except when $q=p^m$, $p\equiv 3\mod 4$, $m$ even. In these exceptional cases, 
$\Omega_n^{\F_{p^m}-\text{Witt}}$ is isomorphic to $W(\Z_p)$, and $w$ is the  homomorphism induced by inclusion $W(\Z_p)\to W(\F_{p^m})$.
\end{proposition}

The proposition follows from Theorem \ref{T: F-Witt}, Corollary \ref{C: 2}, and the commutativity of diagram \eqref{E: commute} for $k,m>0$. In particular, Corollary \ref{C: 2} tells us that the map $w:\Omega_{4k}^{\Z_p-\text{Witt}}\to W(\Z_p)$ is an isomorphism, and so from the commutativity of the diagram, $w:\Omega_{4k}^{\F_{p^m}-\text{Witt}}\to W(\F_{p^m})$ is isomorphic to the homomorphism $W(\Z_p)\to W(\F_{p^m})$. We recall the properties of this homomorphism.

Consider the homomorphism $\psi: W(\Z_p)\into W(\F_{p^m})$ induced by the natural inclusion $\Z_p\into \F_{p^m}$. If $p\equiv 1\mod 4$, then $p^m\equiv 1\mod 4$ for all $m>0$, while if $p\equiv 3\mod 4$, $p^{2m+1}\equiv 3\mod 4$ and $p^{2m}\equiv 1\mod 4$. In all cases except the last ($p^{2m}$ for $p\equiv 3\mod 4$),
$\psi$ is an isomorphism. This follows from the discussion in the proof of Proposition \ref{P: surj}: in these cases we know that $W(\Z_p)\cong W(F_{p^m})$ abstractly and $\psi$ clearly preserves dimension and determinant of the form, which constitute a complete set of invariants.

In the exceptional case, $W(\Z_p)\cong \Z_4$, generated by $\langle 1\rangle$, but $W(\F_{p^{2m}})\cong \Z_2\oplus \Z_2$, in which $\langle 1\rangle$ is one of the generators. So $\psi$ maps $W(\Z_p)$ onto one of the summands of $W(\F_{p^m})$ with kernel $\Z_2$.

What does this mean geometrically? As an example, consider $\Z_3$ and $\F_9\cong \Z_3[x]/\langle x^2-2\rangle$. Let $\{1,x\}$ denote a basis of $\F_9$ as a $2$-dimensional vector space over $\Z_3$. We showed above that   $\C P^2$ is a generator of $\Omega_4^{\Z_3-\text{Witt}}$ corresponding to $\langle 1\rangle$ in $W(\Z_3)$. According to the results of the preceding section, $\Omega_4^{\Z_3-\text{Witt}}\cong W(\Z_3)\cong \Z_4$ so that $\C P^2$ has order $4$. Since $\Omega_4^{\Z_3-\text{Witt}}=\Omega_4^{\F_9-\text{Witt}}$, $\C P^2$ also has order $4$ in the latter group. However $W(\F_9)\cong \Z_2\oplus \Z_2$, so,  in particular, $w(\C P^2\#\C P^2)=0\in W(\F_9)$. To see where a proof of injectivity along the lines of Siegel breaks down, observe that $w(\C P^2\#\C P^2)$ is represented by $\begin{pmatrix}1&0\\0&1\end{pmatrix}$. Since this pairing is nonzero in $W(\Z_3)$, in order to find an isotropic element over $\F_9$, we must mix  $1$ and $x$ nontrivially. For example, let $V_1$ and $V_2$ denote two disjoint spheres in $\C P^2\#\C P^2$ generating $I^{\bar m}H_2(\C P^2\#\C P^2;\F_9)\cong H_2(\C P^2\#\C P^2;\F_9)\cong H_2(\C P^2\#\C P^2)\otimes \F_9$. An isotropic element is  given by $[V_1]+x[V_2]$, since $([V_1]+x[V_2])\cdot([V_1]+x[V_2])=1+x^2=1+2=0\in \F_9$. There is no way to create a cycle representing $[V_1]+x[V_2]$ that is irreducible in the sense discussed in the injectivity  proof in Section \ref{S: finite}, since the mixing of $1$ and $x$ coefficients would prevent us from piping simplices together (for that matter, we cannot make sense of taking $x$ copies of something).

\subsection{$\Z_p$-Witt bordism as a generalized homology theory}\label{S: homology}

This subsection contains a computation of the homology theory $\Omega^{\Z_p-\text{Witt}}_*(\cdot)$. As for rational Witt bordism, $K$-Witt bordism yields a generalized homology theory for any $K$; as noted by Siegel \cite[Chapter IV]{Si83}, this follows from Akin \cite[Proposition 7]{Ak75}, making the obvious generalizations from unoriented to oriented bordism. Akin's axioms are easy to check for $K$-Witt spaces with boundary, by making use of the collars on the boundaries to see that Akin's cuttings and pastings do not create new links that would violate the Witt conditions. 

It follows from the results of the preceding sections that the only parameter that matters in $K$-Witt-bordism is the characteristic of the field $K$. By Siegel, $\Omega^{\Q-\text{Witt}}_*(\cdot)\cong KO[1/2]_*(\cdot )$, so we focus on fields of finite characteristic.  In particular, there is no loss of generality limiting ourselves to $\Z_p$. 
It turns out that $\Omega^{\Z_p-\text{Witt}}_*(\cdot)$ splits as a sum of ordinary homology theories with coefficients, which is the content of the following theorem. 

\begin{theorem}\label{T: EM}
As a homology theory on CW complexes, $\Z_p$-Witt bordism splits as a direct sum of ordinary homology with coefficients. In particular, $$\Omega^{\Z_p-\text{Witt}}_n(X)\cong  \bigoplus_{r+s=n} H_r(X;\Omega^{\Z_p-\text{Witt}}_s).$$ 
\end{theorem}
\begin{proof}
Brown representability tells us that $\Z_p$-Witt bordism is representable by a spectrum; see, e.g. \cite[Theorem 14.35, Corollary 14.36, and Remark 1 on page 331]{SWITZER}. To see that $\Z_p$-Witt bordism on  CW complexes satisfies the wedge axiom, and hence the conditions in \cite[Remark 1]{SWITZER}, note the unreduced version of the wedge axiom within the proof of \cite[Proposition 10.16]{SWITZER} and that our bordism theory  breaks into a sum over what happens on connected components. We let $W_p$ denote the spectrum yielding $\Z_p$-Witt bordism. We will show below that $W_p$ is, in fact, an $MSO$-module spectrum, but the proof of this fact is deferred for now.

Let $X$ be a compact CW complex so that the homology of $X$ is finitely generated. We consider the Atiyah-Hirzebruch spectral sequence for $\Omega^{\Z_p-\text{Witt}}_*(X)$, whose $E^2_{r,s}$ term is $H_r(X;\Omega^{\Z_p-\text{Witt}}_s)$. Since $\Omega^{\Z_p-\text{Witt}}_s$ is $2$-primary for all $s>0$, the only odd torsion in the spectral sequence is in the terms  $H_r(X;\Omega^{\Z_p-\text{Witt}}_0)\cong H_r(X;\Z)$, which lie along the $x$-axis. It follows that the odd torsion part of the spectral sequence splits off and that the  odd torsion subgroup of $\Omega^{\Z_p-\text{Witt}}_n(X)$ is isomorphic to the odd torsion subgroup of $H_n(X)$. 

Now, for the rest of $\Omega^{\Z_p-\text{Witt}}_n(X)$, we consider instead the theory $\Omega^{\Z_p-\text{Witt}}_n(X;\Z_{(2)})\cong \Omega^{\Z_p-\text{Witt}}_n(X)\otimes \Z_{(2)} $ with 2-local coefficients (see \cite[Section II.5]{RUD} for a general reference). In this case, the Atiyah-Hirzebruch $E^2$ terms look like $H_r(X;\Omega^{\Z_p-\text{Witt}}_s\otimes \Z_{(2)})$. Since $\Omega^{\Z_p-\text{Witt}}_s$ is 2-torsion for $s>0$, in this case the terms are identical to those from the spectral sequence before the localization, while the $(r,0)$ term becomes $H_r(X;\Z_{(2)})$. Now, this spectral sequence degenerates at the $E^2$ term by a theorem of Taylor and Williams \cite[Section 2]{TW79} that states that any 2-local MSO-module spectrum (in this case $(W_p)_{(2)}$) splits as a wedge of Eilenberg-MacLane spectra.  
The upshot of this is the degeneration of the Atiyah-Hirzebruch spectral sequence.  By this degeneration, $$\Omega^{\Z_p-\text{Witt}}_n(X)\otimes \Z_{(2)}\cong \Omega^{\Z_p-\text{Witt}}_n(X;\Z_{(2)})\cong H_n(X;\Z_{(2)}) \oplus \bigoplus_{\overset{r+s=n}{s>0}} H_r(X;\Omega^{\Z_p-\text{Witt}}_s).$$

The group $\Omega^{\Z_p-\text{Witt}}_n(X)$ must be finitely generated (since everything else in the spectral sequence is), so, by basic facts about localization, we can read off the $2$-primary and infinite cyclic summands of $\Omega^{\Z_p-\text{Witt}}_n(X)$ from $\Omega^{\Z_p-\text{Witt}}_n(X)\otimes \Z_{(2)}$. In particular, any $2$-primary components are shared between the two, and $\Z$ summands of the former, correspond bijectively with 
 $\Z_{(2)}$ terms in the latter, which can come only from the summand $H_n(X;\Z_{(2)})$. Thus, modulo odd torsion, we must have $\Omega^{\Z_p-\text{Witt}}_n(X)\cong H_n(X;\Z) \oplus \displaystyle\bigoplus_{\overset{r+s=n}{s>0}} H_r(X;\Omega^{\Z_p-\text{Witt}}_s).$
 
But now we know that the odd torsion of $\Omega^{\Z_p-\text{Witt}}_n(X)$ is that of $H_n(X)$. 
So, putting everything together, we must have 
$$\Omega^{\Z_p-\text{Witt}}_n(X)\cong H_n(X;\Z) \oplus \bigoplus_{\overset{r+s=n}{s>0}} H_r(X;\Omega^{\Z_p-\text{Witt}}_s)\cong \bigoplus_{r+s=n} H_r(X;\Omega^{\Z_p-\text{Witt}}_s),$$ 
since $\Omega^{\Z_p-\text{Witt}}_0\cong \Z$.

For infinite CW complexes, we can now use the preceding formula and take direct limits (see \cite[page 331, Remark 1]{SWITZER}).
\end{proof}

Now we return to the following lemma.

\begin{lemma}
$W_p$ is an $MSO$-module spectrum. 
\end{lemma}
\begin{proof}
To begin with, $W_p$ is, in the language of Rudyak \cite[Definition III.7.7]{RUD}, a quasi-module spectrum over the Thom spectrum $MSO$. This means that for CW pairs $(X,A)$ and $(Y,B)$ and letting $\Omega^{SO}_*(\cdot)$ be smooth oriented bordism,  there is a pairing $\Omega^{SO}_*(X,A)\otimes \Omega^{\Z_p-\text{Witt}}_*(Y,A)\to \Omega^{\Z_p-\text{Witt}}_*(X\times Y,X\times B\cup A\times Y)$ that possesses  the expected associativity. To see that such a pairing exists, we observe that smooth oriented manifolds are all $\Z_p$-Witt spaces for any $p$ and so is the product of a smooth oriented manifold with a $\Z_p$-Witt space (the links don't change). Furthermore, we should check that if  $M^m$ is a manifold with boundary and $N^n$ is a $\Z_p$-Witt space with $\Z_p$-Witt boundary, then $\bd (M\times N)=M\times \bd N\cup_{\bd M\times \bd N} \bd M\times N$ is $\Z_p$-Witt. The links in $\text{int}(M)\times \bd N$ are the same as the links of $\bd N$, and the links of $\bd M\times N$ are the same as the links of $N$. Now, if $x\in \bd M\times \bd N$, then $x$ has a distinguished neighborhood in this space of the form $\R^{m-1}\times \R^{n-1-k}\times cL$. In $M\times \bd N$, this neighborhood expands to the form  $\R^{m-1}\times \R^{n-1-k}\times cL\times (-1,0]$, where the $(-1,0]$ coordinate represents the collar of $\bd M$ in $M$. But in $\bd M\times N$, this neighborhood similarly expands to $\R^{m-1}\times \R^{n-1-k}\times cL\times [0,1)$. where $[0,1)$ is the collar of $\bd N$ in $N$. Putting these together, $x$ has a neighborhood of the form $\R^{m+ n-1-k}\times cL$, and so again the link of $x$ in $\bd (M\times N)$ is $L$, which is a link in $\bd N$, and hence satisfies the Witt condition. 

It is not true in general that being a quasi-module spectrum leads to being an actual module spectrum (see \cite{RUD}), but we will be able to show this for $W_p$.

We note that both $MSO$ and $W_p$ have finite $\Z$-type, meaning that they are bounded below (i.e. their coefficient groups vanish below a certain dimension) and each $\pi_i(MSO)$ and $\pi_i(W_p)$ is finitely generated. For $MSO$ this is well-known; see, e.g., \cite[Theorem 18.8]{MS}. For $W_p$, this follows from Theorem \ref{T: F-Witt}. Thus, by \cite[Proposition II.4.26.ii]{RUD}, we may assume, up to equivalence, that $MSO$ and $W_p$ have finite type, i.e. that each skeleton of each spectrum is finite (see \cite[Definition II.1.2.e]{RUD}).

We can now apply a module spectrum version of \cite[Theorem III.7.3]{RUD}. It is noted by Rudyak immediately prior to his \cite[Theorem III.7.8]{RUD} that there is such a module spectrum version of his Theorem III.7.3, which gives conditions for when a quasi-ring spectrum is in fact a ring spectrum. The actual hypotheses of his Theorem III.7.8 as stated are more restrictive than those of Theorem III.7.3 only because this is what is needed in the rest of the book, but a version or Theorem III.7.3 in its full generality can be directly applied to the case of quasi-module spectra.  The quasi-module version of this theorem says the following: Since our spectra have finite $\Z$-type, our MSO-quasi-module structure will be induced by a unique (up to homotopy) MSO-module structure on $W_p$ if the following groups vanish: $\ilim^1\{W_p^{-1}(W_p^{(n)})\}$, $\ilim^1\{W_p^{-1}(MSO^{(n)}\wedge W_p^{(n)})\}$, and $\ilim^1\{W_p^{-1}(MSO^{(n)}\wedge MSO^{(n)}\wedge W_p^{(n)})\}$. To clarify, the superscript $-1$ stands, in each case, for the degree $-1$ generalized cohomology group, while $E^{(n)}$ stands for the $n$-skeleton of the spectrum $E$. The existence of a module morphism $MSO\wedge W_p\to W_p$ inducing the module structure on homology is automatic, but the vanishing of these three groups is necessary to yield, respectively, the multiplicative unit, uniqueness, and associativity for the module spectrum structure.

To show that these groups vanish, we consider the Atiyah-Hirzebruch spectral sequences. Let $\mf X^{(n)}$ be any of $W_p^{(n)}$, $MSO^{(n)}\wedge W_p^{(n)}$, or $MSO^{(n)}\wedge MSO^{(n)}\wedge W_p^{(n)}$. Then, for $W_p^{-1}(\mf X^{(n)})$, the relevant $E_2$ terms of the spectral sequence will have the form $H^i(\mf X^{(n)};\Omega^{\Z_p-\text{Witt}}_{-j})$, with $i+j=-1$. For $j>0$, these terms vanish obviously. For $j=0, i=-1$, they also vanish. To see this, note that since $\Omega^{\Z_p-\text{Witt}}_k=\pi_k(W_p)=0$ and $\Omega^{SO}_k=\pi_k(MSO)=0$ for $k<0$, $MSO$ and $W_p$ are connected spectra (this is just the definition of connected spectrum). By \cite[Proposition II.4.5.iv]{RUD}, for any spectrum $E$, the inclusion $E^{(k)}\subset E$ induces isomorphisms on $\pi_i$ for $i\leq k$, and by  \cite[Lemma II.4.2]{RUD}, we may assume (by replacing spectra with equivalent ones) that, if $E$ is connected, $E^{(k)}=*$ for $k\leq -1$. It follows from these two facts that  $MSO^{(n)}$ and $W_p^{(n)}$ are also  connected spectra. Furthermore, by \cite[Proposition II.4.5.i]{RUD}, the smash products of connected spectra are connected. In particular, then, each of our $\mf X^{(n)}$ is connected. So, by \cite[Corollary II.4.7]{RUD}, $H_i(\mf X^{(n)})=0$ for $i<0$, and, by the universal coefficient theorem for spectra with coefficients in an abelian group \cite[Theorem II.4.9]{RUD}, $H^i(\mf X^{(n)};G)=0$ for any abelian group $G$ and any $i\leq -1$. 

Finally, we consider the terms $H^i(\mf X^{(n)};\Omega^{\Z_p-\text{Witt}}_{-j})$ for $j<0$, $i>-1$. As noted above, we can assume up to equivalence that the skeleta $W_p^{(n)}$ and $MSO^{(n)}$ are finite spectra, and, as a consequence of \cite[Proposition II.1.5.ii,iii]{RUD}, the (co)homology of a finite spectrum is the same as the (co)homology of some finite CW complex. Thus, applying the K\"unneth theorem to the smash products $\mf X^{(n)}$ (see the remarks following \cite[Theorem II.4.11]{RUD}), the remaining terms $H^i(\mf X^{(n)};\Omega^{\Z_p-\text{Witt}}_{-j})$, $i>-1$, look like the cohomology of products of finite complexes with coefficients in the finite groups $\Omega^{\Z_p-\text{Witt}}_{-j}$, $j<0$. Therefore, they must vanish for sufficiently large $i$ and will be finite groups otherwise.

It now follows from the Atiyah-Hirzebruch spectral sequence that each $W_p^{-1}(\mf X^{(n)})$ is finite, and this will suffice to make the above $\ilim^1$ groups vanish by \cite[Corollary III.2.18]{RUD}, which notes that $\ilim^i \mf A=0$ for $i>0$ if $\mf A$ is a system of finite abelian groups.
\end{proof}

\bibliographystyle{amsplain}
\bibliography{bib}

\providecommand{\bysame}{\leavevmode\hbox to3em{\hrulefill}\thinspace}
\providecommand{\MR}{\relax\ifhmode\unskip\space\fi MR }
\providecommand{\MRhref}[2]{%
  \href{http://www.ams.org/mathscinet-getitem?mr=#1}{#2}
}
\providecommand{\href}[2]{#2}
\begin{thebibliography}{10}

\bibitem{Ak75}
Ethan Akin, \emph{Stiefel-{W}hitney homology classes and bordism}, Trans. Amer.
  Math. Soc. \textbf{205} (1975), 341--359. \MR{MR0358829 (50 \#11288)}

\bibitem{Ba02}
Markus Banagl, \emph{Extending intersection homology type invariants to
  non-{W}itt spaces}, vol. 160, Memoirs of the Amer. Math. Soc., no. 760,
  American Mathematical Society, Providence, RI, 2002.

\bibitem{BaIH}
\bysame, \emph{Topological invariants of stratified spaces}, Springer
  Monographs in Mathematics, Springer-Verlag, New York, 2006.

\bibitem{GBF6}
Markus Banagl and Greg Friedman, \emph{Triangulations of $3$-dimensional
  pseudomanifolds with an application to state-sum invariants}, Algebr. Geom.
  Topol. \textbf{4} (2004), 521--542.

\bibitem{Bo}
A.~Borel~et. al., \emph{Intersection cohomology}, Progress in Mathematics,
  vol.~50, Birkhauser, Boston, 1984.

\bibitem{BRTG}
Glen Bredon, \emph{Topology and geometry}, Springer-Verlag, New York, 1993.

\bibitem{GBF17}
Greg Friedman, \emph{Intersection homology and {P}oincar\'e duality on
  homotopically stratified spaces}, submitted.

\bibitem{GBF18}
\bysame, \emph{On the chain-level intersection pairing for {PL}
  pseudomanifolds}, submitted.

\bibitem{GBF3}
\bysame, \emph{Stratified fibrations and the intersection homology of the
  regular neighborhoods of bottom strata}, Topology Appl. \textbf{134} (2003),
  69--109.

\bibitem{GBF11}
\bysame, \emph{Superperverse intersection cohomology: stratification
  (in)dependence}, Math. Z. \textbf{252} (2006), 49--70.

\bibitem{GBF10}
\bysame, \emph{Singular chain intersection homology for traditional and
  super-perversities}, Trans. Amer. Math. Soc. \textbf{359} (2007), 1977--2019.

\bibitem{GoBo}
Mark Goresky, \emph{Witt space cobordism theory}, Intersection Cohomology
  (et.~al. A.~Borel, ed.), Progress in Mathematics, vol.~50, Birkhauser,
  Boston, 1984, pp.~209--214.

\bibitem{GM1}
Mark Goresky and Robert MacPherson, \emph{Intersection homology theory},
  Topology \textbf{19} (1980), 135--162.

\bibitem{GM2}
\bysame, \emph{Intersection homology {II}}, Invent. Math. \textbf{72} (1983),
  77--129.

\bibitem{GS83}
Mark Goresky and Paul Siegel, \emph{Linking pairings on singular spaces},
  Comment. Math. Helvetici \textbf{58} (1983), 96--110.

\bibitem{HUY}
Daniel Huybrechts, \emph{Complex geometry: An introduction}, Springer, New
  York, 2005.

\bibitem{Ki}
Henry~C. King, \emph{Topological invariance of intersection homology without
  sheaves}, Topology Appl. \textbf{20} (1985), 149--160.

\bibitem{KirWoo}
Frances Kirwan and Jonathan Woolf, \emph{An introduction to intersection
  homology theory. second edition}, Chapman \& Hall/CRC, Boca Raton, FL, 2006.

\bibitem{LAM}
T.Y. Lam, \emph{The algebraic theory of quadratic forms}, W.A. Benjamin, Inc.,
  Reading, Massachusetts, 1973.

\bibitem{MV86}
Robert MacPherson and Kari Vilonen, \emph{Elementary construction of perverse
  sheaves}, Invent. Math. \textbf{84} (1986), 403--435.

\bibitem{Mc78}
Clint Mc{C}rory, \emph{Stratified general position}, Algebraic and geometric
  topology (Proc. Sympos., Univ. California, Santa Barbara, Calif. 1977)
  (Berlin), Lecture Notes in Math., vol. 664, Springer, 1978, pp.~142--146.

\bibitem{MH}
J.~Milnor and D.~Husemoller, \emph{Symmetric bilinear forms}, Springer Verlag,
  New York, 1973.

\bibitem{MS}
John~W. Milnor and James~D. Stasheff, \emph{Characteristic classes}, Princeton
  University Press, Princeton, N. J., 1974, Annals of Mathematics Studies, No.
  76. \MR{MR0440554 (55 \#13428)}

\bibitem{MK}
James~R. Munkres, \emph{Elements of algebraic topology}, Addison-Wesley,
  Reading, MA, 1984.

\bibitem{RUD}
Yuli~B. Rudyak, \emph{On {T}hom spectra, orientability, and cobordism},
  Springer Monographs in Mathematics, Springer-Verlag, Berlin, 1998, With a
  foreword by Haynes Miller. \MR{MR1627486 (99f:55001)}

\bibitem{Si83}
P.H. Siegel, \emph{Witt spaces: a geometric cycle theory for {KO}-homology at
  odd primes}, American J. Math. \textbf{110} (1934), 571--92.

\bibitem{SWITZER}
Robert~M. Switzer, \emph{Algebraic topology - homology and homotopy}, Springer,
  Berlin, 2002.

\bibitem{TW79}
Laurence Taylor and Bruce Williams, \emph{Local surgery: foundations and
  applications}, Algebraic Topology, Aarhus 1978 (Proc. Sympos., Univ. Aarhus,
  Aarhus, 1978) (Berlin), Lecture Notes in Math., vol. 763, Springer, 1979,
  pp.~173--207.

\end{thebibliography}

Several diagrams in this paper were typeset using the \TeX\, commutative
diagrams package by Paul Taylor.

\end{document}